\documentclass[reqno,10pt]{amsart} 
\usepackage{amscd,amsfonts,mathrsfs,amsthm,enumerate}
\usepackage{amssymb, amsmath}
\usepackage{calc}            
\usepackage{amscd,amsfonts,mathrsfs,amsthm,enumerate} 

\newcommand \be {\begin{equation}}
\newcommand \ee {\end{equation}}
\newcommand \loc {\text{loc}} 
\newcommand \del {\partial} 
 
\newcommand \eps \epsilon 
\newcommand \Ccal {\mathcal C} 
\def\normal   {\vec n}
\def\flow    {hyperbolic mean curvature flow}  
\newcommand{\nrg}{\textsl e}
\def\vel      {{\sigma}}
\def\tvel     {{S}}
\def\acc      {{\alpha}}
\def\tacc     {{A}}
\def\disp     {\displaystyle}
\def\dtx      #1{\frac{{d}\xi^{#1}}{{d}t}}
\def\real     #1{{\mathbb R^{#1}}}

\def\natural  #1{{\mathbb N^{#1}}}
\def\R     {\mathbb R}

\def\dd       #1#2#3{{#1}_{#2#3}}

\def\dddd     #1#2#3#4#5{{#1}_{#2#3#4#5}}

\def\uu       #1#2#3{{#1}^{#2#3}}

\def\lap      {\Delta }
\def\dt       {{\frac{d}{dt}}}
\def\pt       {{\frac{\partial}{\partial t}}}
\def\ds       {{\frac{d}{ds}}}
\def\ddt      {{\frac{d^2}{dt^2}}}
\def\dn       #1{\nabla_{\hspace{-2pt}#1}\,} 

\newtheorem{theorem}{Theorem}[section]   
\newtheorem{lemma}[theorem]{Lemma}   
\newtheorem{corollary}[theorem]{Corollary}   
\newtheorem{proposition}[theorem]{Proposition}   
\newtheorem{remark}[theorem]{Remark}   
\theoremstyle{definition}   
\newtheorem{definition}[theorem]{Definition}   
   
\numberwithin{equation}{section}   

\newcommand{\bfig}{\begin{figure}}
\newcommand{\efig}{\end{figure}}

\begin{document}
\title
[The hyperbolic mean curvature flow]
{The hyperbolic mean curvature flow}  

\author[P.G. L{\tiny e}F{\tiny loch} and K. S{\tiny moczyk}]{\sc P{\smaller hilippe} G. L{\smaller e}F{\smaller loch} 
and K{\smaller nut} S{\smaller moczyk}}
\address{
Philippe G. LeFloch : Laboratoire Jacques-Louis Lions \& 
Centre National de la Recherche Scientifique, Universit\'e de Paris 6, 
4 Place Jussieu, 75252 Paris, France. 
\newline 
\indent {\it E-mail address :} {\tt LeFloch@ann.jussieu.fr.} 
\newline 
\newline 
\indent Knut Smoczyk :  Institut f\"ur Differentialgeometrie,  Leibniz Universit\"at Hannover, 
Welfengarten 1, 30167 Hannover, Germany.
\newline 
\indent {\it E-mail address :} {\tt Smoczyk@math.uni-hannover.de}
}
\begin{abstract}
We introduce a geometric evolution equation of hyperbolic type, which 
governs the evolution of a hypersurface moving in the direction of its mean curvature vector. 
The flow stems from a geometrically natural action containing kinetic and internal energy terms. 
As the mean curvature of the hypersurface is the main driving factor, 
we refer to this model as the {\sl hyperbolic mean curvature flow} (HMCF).   
The case that the initial velocity field is normal to the hypersurface is of particular interest:
this property is preserved during the evolution and gives rise to a comparatively 
simpler evolution equation. We also consider the case where the manifold can be viewed as a graph 
over a fixed manifold. Our main results are as follows. 
First, we derive several balance laws satisfied by the hypersurface during the evolution. 
Second, we establish that the initial-value problem is locally well-posed in Sobolev spaces;  
this is achieved by exhibiting a convexity property satisfied by the energy density which is naturally associated with the flow. Third, we provide some criteria ensuring that the flow will blow-up in finite time. 
Fourth, in the case of graphs, we introduce a concept of weak solutions 
suitably restricted by an entropy inequality, and we prove that a classical solution is unique in the 
larger class of entropy solutions. In the special case of one-dimensional graphs, a global-in-time existence result is established.  
\end{abstract}

\renewcommand{\subjclassname}{
  \textup{2000} Mathematics Subject Classification}  
\date{\today}   
\maketitle



\section{Introduction}
\label{S-intro}

Our aim in this paper is to introduce and study a geometric 
evolution equation of hyperbolic type which describes the flow  
$$
F:[0,T)\times M\to\real{n+1},\quad T>0 
$$
of an immersed $n$-dimensional hypersurface $M$ in the Euclidean space. 
 We derive this evolution equation  
from a geometrically natural action functional based on the local energy density
$$
\nrg := {1 \over 2} \Big(\Big|{d \over dt} F \Big|^2 + n \Big), 
$$
involving the kinetic energy of the hypersurface and the internal energy 
associated with its volume. The equation under consideration models the nonlinear motion of an elastic membrane, 	
driven by its surface tension only. Our model is purely geometric and
requires no constitutive equation on the membrane material 
(contrary to what is required in the theory of nonlinear elastic bodies or shells). 
As the mean curvature of the hypersurface is the main driving factor, 
we refer to this model as the {\sl hyperbolic mean curvature flow (HMCF);}
see Proposition~\ref{eq of motion} below.    
Stationary solutions of this flow will be minimal hypersurfaces with vanishing
kinetic energy.

The flow equation takes a simpler form in the case that the initial velocity 
is normal to the hypersurface, i.e. if its tangential part vanishes: 
\be
\label{norm}
\left(\frac{dF}{dt}\right)^{\top}_{\big\vert t=0}=0.
\ee
Namely,  
from the momentum conservation law satisfied by a general flow, we will   
deduce that tangential components of the velocity vector vanish for all times if they vanish initially. 
Hence, under this assumption, the (normalized version) of the proposed HMCF equation reads  
\be
\tag{HMCF'}
\aligned 
& \frac{d^2F}{dt^2}=\nrg \,H\nu-\nabla\nrg,
\\
& \left(\frac{dF}{dt}\right)^{\top}_{\big\vert t=0}=0.
\endaligned 
\ee
where the scalar $H$ is the mean curvature of the hypersurface and the vector $\nu$ denotes its 
unit normal (chosen to be inward pointing when $M$ is compact without boundary). 
In fact, the assumption (\ref{norm}) is geometrically motivated in the sense that tangential
variations do not alter the shape of the hypersurface and merely correspond to reparametrizations
by a suitably chosen family of (time-dependent) diffeomorphisms. 
Since, geometrically, only (HMCF') is of interest, we will mainly study this flow, which  
we refer to as the normal mean curvature flow equation. 

The main results established in the present paper are as follows. 
After introducing the proposed flow in Sections~\ref{S-struct} and \ref{S-HMCF}, 
we derive in Section~\ref{S-cons}
several conservation laws or balance laws satisfied by the hyperbolic mean curvature flow. 
Then, in Section~\ref{S-graphs}, we begin our investigation of the properties satisfied by general solutions to the 
hyperbolic flow by 
restricting attention to the important case that the hypersurface is represented as 
an entire graph over $\real{n}$: we prove the local well-posedness of the flow equation, and introduce a concept of 
weak solutions suitably restricted by an entropy inequality; 
we also prove the uniqueness of a classical solution within the class of weak solutions, 
and for one-dimensional graphs we establish the global-in-time existence of weak solutions with bounded variation. The convexity of the 
measure $e \, d\mu$ with respect to certain well-chosen variables is an essential observation for these results. 
Then, for the rest of the paper we turn to the equation (HMCF') for normal flows and, 
in Section~\ref{S-local}, we prove that the equations under consideration can be recast in the form of 
a first-order nonlinear hyperbolic system, and we obtain a local-in-time existence result for the evolution 
of general compact manifolds. Next, in Section~\ref{S-blow-up},
we provide some criteria ensuring that the flow will blow-up in finite time, due to the formation of 
geometric singularities or shock waves. 
For general material on flows by mean curvature we may refer to \cite{brendle,huisken}, and on
 nonlinear wave equations to \cite{hoermander,shatahstruwe,sogge}.


\section{Structure equations for general flows}
\label{S-struct}

Let $F:M \to\real{n+1}$ be a smooth immersion of
an orientable smooth manifold $M$ of dimension $n$ into $\real{n+1}$, and 
let $\nu$ be the unit normal vector defined along the hypersurface and chosen 
to be inward pointing when the manifold is compact without boundary. 
In local coordinates $(x^i)_{i=1,\dots, n}$, we have 
$$
F_i:=\nabla_i F:=dF\left(\frac{\partial}{\partial x^i}\right)
=\frac{\partial F}{\partial x^i}
$$
and 
$$
\nu = {F_1 \wedge \ldots \wedge F_n \over |F_1 \wedge \ldots \wedge F_n|}. 
$$
The induced metric $g=\dd gij\, dx^i\otimes dx^j$ and the second fundamental 
form $h=\dd hij\,dx^i\otimes dx^j$ of the hypersurface are 
$$
\dd gij=\langle F_i,F_j\rangle, 
\qquad 
\dd hij=-\langle F_i,\dn j\nu\rangle,
$$
respectively. Here, $\nabla$ denotes the Levi-Civita connection associated with $g$.  
Throughout, we use Einstein's summation convention on repeated indices 
and, for simplicity, we keep the same notation $\langle\cdot,\cdot\rangle$ 
for both the standard inner product on $\real{n+1}$ and the induced inner product on $M$.   
Latin indices are raised with the inverse $(\uu gij)$ of the metric $(\dd gij)$ so, for instance, 
$$
h_i^j : = h_{ik} g^{kj}.  
$$

We denote by $R_{ijkl}$ the components of the Riemann curvature tensor of the 
hypersurface in local coordinates.
We denote the induced volume form on $M$ by $d\mu$ and the (scalar) mean curvature by $H := \uu gij\dd hij$.
We will use also the following convention: we identify the gradient $\nabla p$ of a function
$p$ on $M$ with its image $dF(\nabla p)=\nabla^ip\, F_i$. 

The following basic properties of these tensor fields are easily checked from their definitions. 

\begin{lemma} 
\label{Gauss}
The Gauss-Weingarten-Codazzi equations of the hypersurface $M$ 
read: 
\begin{eqnarray}
&& \nabla_iF_j = \dd hij\nu,
\label{rel 1.1}
\\
&& \dn i\nu = -h_i^jF_j,
\label{rel 1.2}
\\
&& \dn i\dd hjk = \dn j\dd hik,
\label{rel 1.3}
\\
&& \dddd Rijkl = \dd hik\dd hjl-\dd hil\dd hjk.
\label{rel 1.4}
\end{eqnarray} 
\end{lemma}

In the present paper, we are interested in a {\sl flow of hypersurfaces,} that is, a smooth family of immersions 
$$
F:[0,T)\times M\to \real{n+1},
$$  
so that all of the tensor fields defined above also depend on the time variable $t$. 
We can then define on $M$ some additional (time-dependent) functions $\vel,\acc$ and $1$-form fields  
$\tvel=\tvel_idx^i,\tacc=\tacc_idx^i$ by
$$ 
\aligned 
& \vel := \left\langle\dt F,\nu\right\rangle,
\qquad 
&&  \tvel_i := \left\langle\dt F,F_i \right\rangle,
\\
& \acc := \left\langle \ddt F,\nu \right\rangle,
\qquad 
&& \tacc_i := \left\langle\ddt F,F_i \right\rangle.
\endaligned
$$
We refer to $\vel$ and $\tvel$ as the {\sl normal} and {\sl tangential velocity} components,
respectively, 
and to $\acc$ and $\tacc$ as the {\sl normal} and {\sl tangential acceleration} components, respectively. 
We have the decomposition 
$$
\dt F=\vel \, \nu+\tvel^iF_i,
\qquad
\ddt\, F=\acc \, \nu+\tacc^iF_i.
$$
 
To express the structure equations satisfied by a general flow, 
it is convenient to introduce the {\sl local energy density} of the hypersurface 
\be 
\label{local energy}
\aligned 
\nrg 
:= & \frac{1}{2}\left(\left|\dt F\right|^2 + \, \big| \nabla F \big|^2 \right)
\\
= & \frac{1}{2} \left(\left|\dt F\right|^2 + n \right),
\endaligned 
\ee 
where we have used $\big| \nabla F \big|^2 = g^{ij}\langle \nabla_iF, \nabla_j F\rangle = g^{ij} g_{ij} = n$.  
A simple computation based on Lemma~\ref{Gauss} then yields the following expressions for the 
components of the velocity and acceleration fields.

\begin{lemma}
\label{lemma 2.2}
Every flow $F$ satisfies the following structure equations:
\begin{eqnarray}
&& \dn i\vel = \left\langle\dt F_i,\nu\right\rangle - h_i^j \tvel_j,
\label{rel 1}
\\
&& \nabla_i\tvel_j = \left\langle\dt F_i,F_j\right\rangle+\vel\dd hij,
\label{rel 2}
\\
&& \acc = \dt\,\vel+\left\langle\nabla\vel,\tvel\right\rangle+h(\tvel,\tvel),
\label{rel 3}
\\
&& \tacc = \dt\,\tvel-d\nrg,
\label{rel 4}
\end{eqnarray}
where $d\nrg = \nabla_j\nrg\, dx^j$ denotes the exterior differential of the function $e$, dual to the gradient
$\nabla\nrg=\nabla^i\nrg\, F_i$.
\end{lemma}

The following notion will be of special interest in this paper.

\begin{definition}
A flow  
$F:[0,T)\times M\to \real{n+1}$ is called a {\sl normal flow} if and only if 
its tangential velocity vanishes identically, that is, $\tvel(t) \equiv 0$ for all 
$t\in[0,T)$. 
\end{definition}

\begin{proposition}~

\begin{enumerate}[1.]
\item
A flow $F$ is normal if and only if its tangential velocity and tangential acceleration 
satisfy
$$
\aligned 
& \tvel(0) = 0 \quad \text{ at the initial time,}
\\ 
& \tacc = -\,d\nrg \quad \text{ at all times.} 
\endaligned 
$$ 

\item
Given a general flow  
$F:[0,T)\times M\to \real{n+1}$,  
there always exists a smooth family of time-dependent diffeomorphisms $\Psi_t: M \to M$
such that the modified flow given by
$$
\aligned 
& \widetilde F : [0,T)\times M \to\real{n+1},
\qquad 
\widetilde F(t,x) := F(t,\Psi_t(x)) 
\endaligned 
$$ 
is a normal flow. In particular, the hypersurfaces 
$$
M_t:=F(t,M), \qquad 
\widetilde M_t:=\widetilde F(t,M)
$$
coincide for each 
$t\in[0,T)$.
\end{enumerate}
\end{proposition}

\begin{proof} The first claim follows immediately from equation \eqref{rel 4}. 
To derive the second claim we consider a general flow 
$$
\dt\,F(t,x) = \sigma(t,x) \, \nu(t,x) + S(t,x).
$$
Since $S(t,x)$ is {\sl tangential} to $M_t$, we can introduce the solution $\Psi_t: M \to M$ be
 the following ordinary differential equation (ODE)
$$
\dt\, \Psi_t(x) = - S(t,\Psi_t(x)).
$$
Then, by setting $S(t,x) =: S^i(t,x) \, F_i(t,x)$ 
we see that the map 
$$
\widetilde F(t,x):=F(t,\Psi_t(x))
$$
satisfies the evolution equation
$$
\aligned 
\dt\,\widetilde F(t,x)
& = \sigma(t,\Psi_t(x)) \, \nu(t,\Psi_t(x)) + S(t,\Psi_t(x)) + F_i(t,\Psi_t(x))\dt\Psi^i_t(x)
\\
& = \sigma(t,\Psi_t(x)) \, \nu(t,\Psi_t(x))
\\
& =: \widetilde\sigma(t,x) \, \widetilde\nu(t,x).
\endaligned 
$$ 
\end{proof} 

The proposition above shows that a general flow $F$ and its normalized version 
$\widetilde F$ can be identified geometrically. Therefore, without genuine restriction, 
our analysis will often be focused on normal flows, which have the general form 
\be
\label{general}
\ddt\, F = \acc \, \nu - \nabla\nrg. 
\ee
It should be observed that, at this stage, the normal acceleration $\acc$ has not been defined yet. 
Our results will show that by prescribing this scalar field (in the forthcoming section) the evolution of the hypersurface
is uniquely determined.

To conclude this section, in view of the computations done in Huisken \cite{huisken} for the standard mean curvature flow, we obtain 
the following first-order evolution equations for the induced metric, volume form, 
second fundamental form, and mean curvature of the hypersurface. 

\begin{lemma}\label{lemma 1} 
The evolution of the tensor fields $\dd gij, d\mu, \nu, \dd hij, H$
associated with a general flow  
$F:[0,T)\times M\to \real{n+1}$ is determined by the equations 
\begin{eqnarray}
&& \dt\dd gij = -2\vel\dd hij+\nabla_i\tvel_j+\nabla_j\tvel_i,
\label{evol 2}
\\
&& \dt d\mu = ({d}^\dagger\tvel-\vel H){d}\mu,
\label{evol 3}
\\
&& \dt\nu = -(\nabla^i\vel+\uu hik\tvel_k)F_i,
\label{evol 4}
\\
&& \dt\dd hij =  \nabla_i\nabla_j\vel-\vel h_i^k\dd hkj+h_i^k\nabla_j\tvel_k
+ h_j^k\nabla_i\tvel_k+\nabla^k\dd hij\tvel_k,
\label{evol 5}
\\
&& \dt H = \lap \vel+\vel|h|^2+\tvel^i\nabla_iH, 
\label{evol 6}
\end{eqnarray}
where ${d}^\dagger\tvel=\nabla^i\tvel_i=\nabla_iS^i$ denotes the divergence of a vector field, 
and $|h|^2 := \uu hij\dd hij$ denotes the (squared) norm of a $2$-tensor field.
\end{lemma}

In particular, if the flow is normal we take $S=0$ and $A = -\, de$ in Lemma~\ref{lemma 1}
and obtain 
\begin{eqnarray}
&& \dt\,\dd gij = -2\vel\dd hij,
\label{evol 277}
\\
&& \dt\, d\mu = -\vel H {d}\mu,
\label{evol 377}
\\
&& \dt\,\nu = - \nabla^i\vel F_i,
\label{evol 477}
\\
&& \dt\,\dd hij =  \nabla_i\nabla_j\vel-\vel h_i^k\dd hkj, 
\label{evol 577}
\\
&& \dt\, H = \lap \vel+\vel|h|^2.  
\label{evol 677}
\end{eqnarray}


\section{The hyperbolic mean curvature flow} 
\label{S-HMCF} 

We are now in a position to introduce the evolution equation that we propose in this paper.  
The flow is going to be defined from an Hamiltonian principle based on a geometrically natural action,
consisting of a kinetic term  
and an internal energy term, which   
is defined geometrically as the local volume density of the hypersurface.

More precisely, let $F:[0,T]\times M \to\real{n+1}$ be a smooth family of immersions of
an orientable manifold $M$ of dimension $n$ into $\real{n+1}$.
Define the total {\sl kinetic energy} at the time $t$ by
$$
K(t) := \int_M  \frac{1}{2} \left|\frac{{d}}{{d}t} F\right|^2{d}\mu_t
$$
and, after integrating over the time interval $[0,T]$, consider 
the action
\be 
\label{act 1}
J_K(F) := \int_{0}^{T}\int_M
       \frac{1}{2} \left|\frac{{d}}{{d}t} F\right|^2{d}\mu_tdt.
\ee 
Define the total {\sl internal energy} of the hypersurface at the time $t$ by 
$$
V(t) := \int_M  \frac{1}{2} |\nabla F|^2 \, {d}\mu_t
=\frac{n}{2}\int_M{d}\mu_t, 
$$
solely determined by the induced volume form, and consider the corresponding action
\be
\label{act 2}
J_V(F) := \frac{n}{2}\int_{0}^{T}\int_M{d}\mu_tdt.
\ee

According to the Hamiltonian principle, we impose that the evolution of the hypersurface 
is stationary for the action $J_V - J_K$,  that is, 
\be 
\label{hamilton}
\ds(J_V - J_K)(F+s\Phi)\bigl\vert_{s=0} = 0 
\ee 
for all $\Phi\in C_0^\infty([0,T]\times M, \real{n+1})$ (compactly supported maps that are 
differentiable oat any order).

\begin{remark} 
Obviously, if the volume of the manifold is infinite, the functionals $J_K(F)$ and $J_V(F)$ above are
only formally defined. This difficulty can be easily overcome by restricting attention  
to any compact subset of $M$. However, since the stationarity condition \eqref{hamilton} 
implied by the Hamiltonian principle itself is formulated
in terms of compactly supported variations, this is unnecessary. 
\end{remark} 

We now show:

\begin{proposition}[Hyperbolic mean curvature flow equation]
\label{eq of motion}
The stationary solutions of the action functional $J_V - J_K$ satisfy 
the equation of motion
\be \tag{HMCF}
\aligned 
	{d^2 \over dt^2}  F =& \, \alpha \, \nu + A^k F_k,
	\\
	\alpha :=& (\nrg-|\tvel|^2)H-\vel {d}^\dagger\tvel,
	\\
	A^k :=& (\vel H-{d}^\dagger\tvel)\tvel^k-\nabla^k\nrg.  
\endaligned 
\ee
\end{proposition}

For instance, if the initial velocity is normal, i.e. if $S_{|t=0}=0$, then it will
follow from the conservation of momentum that $S_{|t}=0$ for all $t$, so that
the HMCF equation reduces to the much simpler equation
\be 
\tag{HMCF'}\label{HMCF'}
\aligned 
& \frac{d^2}{dt^2} F =\nrg \,H\nu-\nabla\nrg,
\\
& \left(\frac{d}{dt} F\right)^{\top}_{\vert t=0}=0.
\endaligned
\ee 
Comparing with the general flow equation \eqref{general} we see that the minimal action 
principle  allows us to identify the normal acceleration, as a linear function in the 
mean curvature $H$: 
$$
\alpha =  \nrg \, H. 
$$
In consequence, for normal flows the evolution of the normal component of the velocity is
proportional to the mean curvature  
$$
{d \over dt}\, \sigma = e(\sigma) \, H. 
$$
Later in this paper, we will prove that the equation (HMCF') is hyperbolic. We will not treat here 
the general system (HMCF) except in the next section where we will derive some (a-priori)
conservation laws for the general flow. The subsequent sections are entirely devoted to
the normal equation (HMCF'), since this (as already noted in the introduction)
is sufficient from the geometric point of view.

Observe also that $\nabla\nrg$ depends on second-order derivatives of $F$, namely
on mixed derivatives in space and time. In agreement with the standard mean curvature flow 
(which is parabolic), the acceleration $\acc$ is defined in terms of the curvature of the manifold $M$, 
and $\acc$, considered as an operator on $M$, is elliptic.  
Therefore, this justifies to refer to the proposed flow
as the {\sl hyperbolic mean curvature flow}. The following sections will show, both, some similarities
and some marked differences between the parabolic and hyperbolic versions of the mean curvature flow. 

\begin{proof}  
A simple computation yields  
$$
\aligned 
\ds J_V(F+s\Phi)\bigl\vert_{s=0}
& = \frac{n}{2}\int_0^T\int_M\uu gij
\langle F_i,\Phi_j\rangle {d}\mu_tdt
\\
& = -\frac{n}{2}\int_0^T\int_M
H\langle \nu,\Phi\rangle {d}\mu_tdt\label{var V},
\endaligned 
$$ 
where, for the second identity, we have integrated by parts and used 
the contracted Gauss formula $\lap F=H\nu$ (a consequence of \eqref{rel 1.1}).

On the other hand, for $J_K$ we obtain the first variation formula 
$$ 
\ds J_K(F+s\Phi)\bigl\vert_{s=0}
 = 
 \int_0^T\int_M\left(
\left\langle\dt F,\dt\Phi\right\rangle+(\nrg-\frac{n}{2})\uu gij
\langle F_i,\Phi_j\rangle\right){d}\mu_tdt,
$$
in which we now successively integrate by parts each term of the right-hand side. 
For the first term we find 
$$
\aligned 
& \int_0^T\int_M
\left\langle\dt F,\dt\Phi\right\rangle {d}\mu_tdt
\\
&=  \int_0^T\dt
\left(\int_M\left\langle\dt F,\Phi\right\rangle {d}\mu_t\right)dt
 -\int_0^T\int_M\left\langle\ddt F,\Phi
\right\rangle {d}\mu_t dt
\\
& \qquad 
-\int_0^T
\left(\int_M\left\langle\dt F,\Phi\right
\rangle \dt {d}\mu_t\right) dt
\\
& = \int_0^T\int_M\left\langle
-\ddt F+(\vel H-{d}^\dagger\tvel)\dt F,\Phi\right\rangle {d}\mu_t dt,
\endaligned  
$$
where we used \eqref{evol 3}. For the second term we obtain
$$ 
\int_0^T\int_M
(\nrg-\frac{n}{2})\uu gij\langle F_i,\Phi_j\rangle {d}\mu_tdt
= -\int_0^T\int_M
\langle\nabla\nrg+(\nrg-\frac{n}{2}) H\nu,\Phi\rangle {d}\mu_tdt. 
$$
Combining the above identities together, we deduce that 
$$ 
\ds(J_V - J_K)(F+s\Phi)\bigl\vert_{s=0}
=\int_{0}^{T}\int_M\langle P,\Phi\rangle {d}\mu_tdt 
$$ 
 with 
$$
\aligned 
P :=&  \ddt F-(\vel H-{d}^\dagger\tvel)\dt F+\nabla\nrg + (\nrg-n)H\nu
\\
=&  \ddt F+\left(\frac{1}{2}(|S|^2-\vel^2 -n)H+\vel {d}^\dagger\tvel\right)\nu
       +\left(\nabla^k\nrg-(\vel H-{d}^\dagger\tvel)\tvel^k\right)F_k.
\endaligned 
$$
Since $\Phi\in C_0^\infty([0,T]\times M, \real{n+1})$ is arbitrary this completes the derivation of (HMCF).
\end{proof}

\section{Conservation laws and balance laws}
\label{S-cons} 

In this section we derive various conservation laws satisfied by solutions of the 
hyperbolic mean curvature flow (HMCF), and we show that, as far as the geometry of the hypersurface
is concerned, one may work within the class of normal flows (HMCF'). 
Consider a family of immersions $F:[0,T) \times M\to\real{n+1}$ satisfying 
the hyperbolic mean curvature equation (HMCF). In the present section, 
all tensor fields under consideration are assumed to be sufficiently smooth. 
We use the so-called $abc$ method discussed by Shatah and Struwe in \cite{shatahstruwe}, and
 multiply (HMCF) by an expression of the general form
$$
a\dt F+b\cdot \nabla F+c F, 
$$
by choosing the variable coefficients $a$, $b$, and $c$ so that higher-order terms in 
the corresponding evolution equations admit a divergence form.

The following lemma shows that there exists a divergence-type form for (HMCF), which will be useful 
for the derivation of conservation laws (modulo lower-order terms). 

\begin{lemma}[A general identity] 
\label{vectorfield}
Let $Y=Y(t,y)$ be a time dependent vector field on $\real{n+1}$.
Every solution of (HMCF) satisfies 
\begin{eqnarray}
\dt\left(\left\langle\dt F,Y\right\rangle{d}\mu_t\right)
&=&\left( \nabla^k\Bigl((n-\nrg)\langle F_k,Y\rangle\Bigr)
+(\nrg-n)\uu gijDY(F_i,F_j)
\phantom{\dt}\right.\nonumber\\
&& \left.+\,DY\left(\dt F,\dt F\right)+\left\langle\dt F,Y_t\right\rangle \right) 
\, {d}\mu_t,
\label{vect}
\end{eqnarray}
where $Y_t=\frac{\partial}{\partial t}Y$ and $DY$ is the spatial differential of $Y$, i.e. 
$$DY=
\dd\delta\alpha\gamma\frac{\partial Y^\gamma}{\partial y^\beta}\,dy^\alpha\otimes dy^\beta.
$$
\end{lemma}

\begin{proof}
We multiply (HMCF) by $Y$ and compute
\begin{eqnarray}
&&\dt\left(\left\langle\dt F,Y\right\rangle{d}\mu_t\right)\nonumber\\
& = &\left(  \Bigl((\nrg-|\tvel|^2)H-\vel{d}^\dagger\tvel\Bigr)
\langle\nu,Y\rangle\,+\,\Bigl((\vel H-{d}^\dagger\tvel)\tvel^k-\nabla^k\nrg
\Bigr)\langle F_k,Y\rangle\phantom{\dt}\right.\nonumber\\
&& \left.+ \left\langle\dt F,Y\right\rangle({d}^\dagger\tvel-\vel H)
  +\,DY\left(\dt F,\dt F\right)+\left\langle\dt F,Y_t\right\rangle \right) \, {d}\mu_t.\nonumber
\end{eqnarray}
Since $\dt F=\sigma\nu+S^kF_k$ we obtain 
$$
\aligned 
& \dt\left(\left\langle\dt F,Y\right\rangle{d}\mu_t\right)
\\
& = \left( \bigl(\nrg-|\tvel|^2-\vel^2\bigr)H\langle\nu,Y\rangle
      -\nabla^k\nrg\langle F_k,Y\rangle 
       + DY\left(\dt F,\dt F\right)+\left\langle\dt F,Y_t\right\rangle \right) 
\, {d}\mu_t,
\\
& = \left( \bigl(n-\nrg\bigr)H\langle\nu,Y\rangle
       -\nabla^k\nrg\langle F_k,Y\rangle 
    + DY\left(\dt F,\dt F\right)+\left\langle\dt F,Y_t\right\rangle \right) 
\, {d}\mu_t.
\endaligned
$$
From $\nabla^kF_k=\Delta F=H\nu$ we get 
$$\bigl(n-\nrg\bigr)H\langle\nu,Y\rangle
-\nabla^k\nrg\langle F_k,Y\rangle
=\nabla^k\Bigl((n-\nrg)\langle F_k,Y\rangle\Bigr)
-(n-\nrg)\uu gijDY(F_i,F_j).
$$
Inserting this in the above expression for 
$\dt\left(\left\langle\dt F,Y\right\rangle{d}\mu_t\right)$ we arrive at \eqref{vect}. 
\end{proof}

From Lemma~\ref{vectorfield} we now derive various conservation laws or balance laws of interest. 
The first result below is a consequence of the invariance of (HMCF) under isometries of $\real{n+1}$.  

\begin{proposition}[Local continuity equation]
\label{continuity}
Every solution of (HMCF) satisfies 
\begin{eqnarray}
\dt\left(\left\langle\dt F,Y\right\rangle{d}\mu_t\right)
&=&\nabla^k\bigl((n-\nrg)\langle F_k,Y\rangle\bigr){d}\mu_t,
\label{cont}
\end{eqnarray}
where $Y=Y^\alpha\frac{\partial}{\partial y^\alpha}$ is any time independent
Killing vector field on $\real{n+1}$.
\end{proposition}

\begin{proof}
Let $Y=Y^\alpha\frac{\partial}{\partial y^\alpha}$
be a Killing vector field on $\real{n+1}$, that is, 
$Y$ generates an isometry
on $\real{n+1}$. Since $Y$ is a Killing vector field, $DY$ is skew-symmetric, so that 
the last three terms in (\ref{vect}) vanish and 
we get \eqref{cont}. 
\end{proof}

\begin{proposition}[Local momentum equation] 
\label{62}
Every solution of (HMCF) satisfies the 
balance law
$$
\dt \left(\left\langle\dt F, F \right\rangle \, d\mu \right)
=
\nabla^k\bigl((n-\nrg)\langle F,F_k\rangle\bigr) \, d\mu
+\bigl((n+2) \, \nrg - n(n+1)\bigr) \, d\mu.
$$ 
\end{proposition} 

\begin{proof} 
We apply Lemma \ref{vectorfield} to the vector field $Y(y,t):=y$. 
Then $DY(v,w)=\langle v,w\rangle$ and $Y_t=0$. Moreover $Y(F(x,t),t)=F(x,t)$, so
that (\ref{vect}) becomes
\begin{eqnarray}
\dt\left(\left\langle\dt F,F\right\rangle{d}\mu_t\right)
&=&\left( \nabla^k\Bigl((n-\nrg)\langle F_k,F\rangle\Bigr)
+(\nrg-n)\uu gij\langle F_i,F_j\rangle
\phantom{\dt}\right.\nonumber\\
&& \left.+\,\left\langle\dt F,\dt F\right\rangle\right) 
\, {d}\mu_t\nonumber\\
&=&\left( \nabla^k\Bigl((n-\nrg)\langle F_k,F\rangle\Bigr)
+(\nrg-n)n+2\nrg-n\right) \, {d}\mu_t,
\nonumber
\end{eqnarray}
which establishes the desired identity.
\end{proof}

We next turn to the component $S$ of the velocity vector. 

\begin{proposition}[Local tangential velocity equation] 
Every solution of (HMCF) satisfies the conservation law 
\be 
\label{momentum}
\dt\bigl(\tvel(X) {d}\mu_t\bigr)=0, 
\ee 
where $X\in\Gamma(TM)$ is any time-independent vector field on $M$.
\end{proposition}

\begin{proof} Given a time-independent vector field $X=X^i\frac{\partial}{\partial x^i}$ defined on $M$,
from \eqref{rel 4} we obtain  
$$
\aligned 
\dt\bigl(\tvel(X)\bigr)
& = {d}\nrg(X)+\left\langle\ddt F,{d}F(X)\right\rangle,
\\
& = (\vel H-{d}^\dagger\tvel)\bigl(\tvel(X)\bigr),
\endaligned 
$$
where the second identity follows by multiplying (HMCF) by $X^iF_i={d}F(X)$. 
The desired conclusion is now clear in view of \eqref{evol 3}.
\end{proof}

The momentum conservation law has the following important consequence.

\begin{corollary}[Reduction to normal flows] 
\label{reduc}
Within the class of flows $F:[0,T) \times M\to\real{n+1}$ whose velocity vector is initially normal 
to the hypersurface, i.e. 
$$
F(0,x) = F_0(x), \qquad \disp \dt F(0,x)=f(x)\nu(0,x).
$$
where $F_0 : M \to\real{n+1}$ is a immersion and $f : M \to \real{}$ a function, 
the following two properties hold:

\begin{enumerate} 
\item
the flow $F$ is a solution to (HMCF) if and only if it is a solution of the normal flow equation (HMCF'). 

\item
any solution of (HMCF') satisfies $S=0$ for all $t\in[0,T)$; in other words,
there exists a family of functions $\vel:[0,T)\times M\to\real{}$ such that
$\vel(0,x)=f(x)$ and 
\be \label{normal vel}
\dt F(t,x)=\vel(t,x) \, \nu(t,x).
\ee 
\end{enumerate}
\end{corollary}

\begin{proof} For each compactly supported tangent vector field $X\in\Gamma_0(TM)$, 
by defining the total tangential momentum in the direction $X$ as 
\be 
\label{total momentum}
p(t,X):=\int_M\tvel(X) {d}\mu_t,  
\ee
it is clear that 
$$
\tvel\bigl\vert_{t}=0\quad \text{ if and only if } \quad p(t,X)=0, \qquad X\in\Gamma_0(TM).
$$
However, the identity (\ref{momentum}) implies 
$$
p(t,X) = p(0,X) = 0, \qquad   t \in [0,T], 
$$
and
$$
\tvel\bigl\vert_{0}=0 \quad \Rightarrow \quad \tvel\bigl\vert_{t}=0,\quad t\in[0,T]. 
$$
Then, we obtain (HMCF') by inserting $\tvel=0$ into (HMCF).
\end{proof}

We continue our derivation of conservation laws satisfied by the hyperbolic mean curvature flow.

\begin{proposition}[Local energy identity] 
\label{energyid}
Every solution of (HMCF) satisfies the conservation law 
\be 
\label{lnrg}
\dt\,\bigl(\nrg\, {d}\mu\bigr) = {d}^\dagger\bigl((n-\nrg)\tvel\bigr){d}\mu.
\ee 
In particular, if the initial velocity is normal along the hypersurface, then
$e \, d\mu$ is conserved along the flow, 
\be 
\label{slnrg}
\dt\,\bigl(\nrg\, {d}\mu\bigr) = 0
\ee 
and $\nrg\, d\mu$ can be seen as a fixed
volume form on $M$.
\end{proposition}

\begin{proof} We multiply (HMCF) by $dF/dt$ and obtain 
$$
\aligned 
\dt\,\nrg
& = \Big\langle\dt F,\left((\nrg-|\tvel|^2)H-\vel {d}^\dagger\tvel\right)
\nu+\left((\vel H-{d}^\dagger\tvel)\tvel^k-\nabla^k\nrg\right)F_k
\Big\rangle
\\
& =  \vel\bigl((\nrg-|\tvel|^2)H-\vel {d}^\dagger\tvel\bigr)
-\nabla^k eS_k
+(\vel H-{d}^\dagger\tvel)|\tvel|^2
\\
& = \vel H\nrg +(n-2\nrg) {d}^\dagger\tvel-\nabla^k eS_k
\\
& = (\vel H-{d}^\dagger\tvel)\nrg +{d}^\dagger\bigl((n-\nrg)\tvel\bigr). 
\endaligned 
$$ 
Then, the desired identity again follows from \eqref{evol 3}. The second statement follows
from the reduction principle in Corollary \ref{reduc}.
\end{proof}

From the above proposition the following global result follows. 

\begin{corollary}[Global conservation laws] 
If $F:[0,T)\times M\to\real{n+1}$ be a solution of (HMCF')
and $M$ is a compact manifold without boundary, then the total energy $E(t)$ defined by
$$
E_M(t):=\int_M \nrg\, {d}\mu_t
=\frac{1}{2}\int_M\left|\frac{{d} F}{{d}t}\right|^2{d}\mu_t
+\frac{n}{2}\int_M{d}\mu_t 
$$
is conserved: 
$$
E_M(t) = E_M(0), \qquad t\in[0,T). 
$$
\end{corollary}

We also state here a compatibility between derivatives of $F$ in time and in space which follows immediately from the fact 
that the derivatives $d/dt$ and $\nabla$ commute.  

\begin{proposition}[Compatibility relation between time and space derivatives]
Every flow $F:[0,T)\times M\to\real{n+1}$ satisfies the conservation law
$$
{d \over dt} F - \nabla \left( \sigma \, \nu + S^j F_j \right) = 0. 
$$
\end{proposition} 

\begin{remark}
There are other globally conserved quantities, which are conserved for 
topological reasons and not because of the special nature of our flow. We give two examples.
\begin{enumerate}[1.]
\item
If $Y$ is a divergence free vector field in $\real{n+1}$ and $M$ is closed, then by Stokes
theorem
$$\int\limits_M\langle\nu,Y\rangle d\mu_t=0.$$ 
\item
The quantity
$$\int\limits_M\frac{\langle F-p,\nu\rangle}{|F-p|^{r+1}}\,d\mu_t$$
with $p\in\real{n+1}$ arbitrary (such that $F(t,x)\neq p$ for all $(t,x)\in [0,T)\times M$)
is the degree of the ``winding map" and hence an invariant.
\end{enumerate}
\end{remark}


\section{Evolution of entire graphs}
\label{S-graphs}

In the rest of this article we restrict our attention to the normal hyperbolic mean curvature flow
(HMCF'). 

\subsection*{Local well-posedness result for the normal flow of graphs} 

It is convenient to begin our investigation with the case of graphs. That is, 
in this section we discuss the case, where $F:[0,T) \times M\to\real{n+1}$
satisfies (HMCF') and such that each $M_t:=F(t,M)$ is an entire graph over a flat subspace 
$Z^\perp\subset\real{n+1}$, where $Z^\perp$ denotes the orthogonal complement
of a unit vector $Z\in\real{n+1}$. 

Without loss of generality, we assume that
$F$ is given by 
$$
F(t,x)=\Bigl(x(t),u\bigl(t,x(t)\bigr)\Bigr)
$$ 
for a time-dependent family of height functions
$u:[0,T) \times M\to\real{}$ and a family of diffeomorphisms $x(t)=\Bigl(x^1(t),\dots,x^n(t)\Bigr)$ 
defined on the hyperplane 
$$
M = Z^\perp = \big\{(x^1,\dots,x^{n+1})\in\real{n+1}:x^{n+1}=0\big\}.
$$
In such a situation, a solution of (HMCF') is completely determined by the time-dependent function $u$. 
From now on, we will use the notation $u_{tt}:=\frac{\partial ^2u}{\partial t^2}$ and 
$u_{tj}:=\frac{\partial ^2u}{\partial t\partial x^j}$. 

\begin{theorem}[Local well-posedness for the normal flow of graphs] 
\label{P-graph}~

\begin{enumerate}[1.]
\item
The hyperbolic mean curvature flow for graphs over a flat hypersurface $Z^\perp$ takes the form of the 
following second-order hyperbolic equation: 
\be 
\label{eq graph}
- u_{tt} 
       +\Bigl(e(\sigma) \, \uu gij + \sigma^2(\uu gij-\uu\delta ij)\Bigr) \dd uij 
       + 2 \, \frac{\sigma}{w}\,\uu\delta ij u_iu_{tj} = 0,
\ee
where 
$$w:=\sqrt{1+|Du|^2}=\sqrt{1+\delta^{ij}u_iu_j},\quad\vel = u_t/w$$
and
$$\uu gij=\uu\delta ij- w^{-2}\uu\delta ik\uu\delta jl u_ku_l.$$

\item
The HMCF equation for graphs can be recast in the form of 
a nonlinear hyperbolic system of $n+1$ equations in the unknowns $\sigma$ and $b=Du=(u_i)_{i=1,\dots,n}$
\be 
\label{eq graph2}
\aligned 
& {\del \sigma \over \del t} 
       - {\del \over \del x^j}  \Big( e(\sigma) {b_i \over w} \, \delta^{ij} \Big) = 0,
\\
& {\del b_i \over \del t} - {\del \over \del x^i} (\sigma w) = 0,
\endaligned 
\ee
which, moreover,
 has a conservative form and is
endowed with the mathematical entropy function  
$$
E(\sigma, b) = {1 \over 2} (\sigma^2 +n) \, \sqrt{1 + \delta^{ij} b_i b_j}.  
$$ 
Moreover, this function is strictly convex provided 
$$
|Du|^2 = \delta^{ij} b_i b_j < {1 \over 2}. 
$$

\item
Furthermore, the graph equation \eqref{eq graph}
is locally well-posed in the following sense: given data $(\sigma, Du)$ prescribed
at the initial time $t=0$ and belonging to the Sobolev space $H^s(\R^n)$ (that is, 
whose all their $s$-order derivatives are squared integrable) for some $s > 1+n/2$,
 there exists a classical solution 
$$
u:[0,T) \times \R^n \to \R \quad \text{ in } L^\infty([0,T), H^{s+1}(\R^n)) \cap Lip([0,T), H^s(\R^n)),
$$
defined on a maximal time interval $[0,T)$. 
\end{enumerate}
\end{theorem}

\begin{proof} {\sl Step 1: Derivation of the HMCF' graph equation.} We will use that 
\be 
\label{non 1}
\dt F=\left(\dt x,\dt u\right)=\left(\dt x,\pt u+u_i\dt x^i\right)=\vel\nu.
\ee 
On the other hand, since 
$$
F_i=\frac{\partial}{\partial x^i}+u_i\frac{\partial}{\partial u},
$$
we obtain
$$
\nu=\frac{1}{w}\left(-\delta^{ij}u_i\frac{\partial}{\partial x^j}+
\frac{\partial}{\partial u}\right)
$$
with 
$w:=\sqrt{1+|Du|^2}=\sqrt{1+\delta^{ij}u_iu_j}$. 
Inserting this into (\ref{non 1}) gives the two equations
\begin{eqnarray}
\dt x^i&=&-\frac{\vel}{w}\,\delta^{ij}u_j,
\label{non 2}
\\
\pt u+u_i\dt x^i&=&\frac{\vel}{w}.
\label{non 3}
\end{eqnarray}
Inserting (\ref{non 2}) into (\ref{non 3}) yields 
\be 
\label{non 4}
\pt u=\vel w.
\ee 
Differentiating (\ref{non 4}) gives
\begin{eqnarray}
\dt\left(\pt u\right)
&=&\frac{\partial ^2u}{\partial t^2}
+\frac{\partial ^2u}{\partial x^i\partial t} \dt x^i
\label{non 5}
\\
&=&w\dt\vel +\vel\dt w.\nonumber
\end{eqnarray}
Now, we can either compute directly or use equation (\ref{rel 3}), $S=0$, 
$\acc=\nrg H$ to see that
\be \label{non 6}
\dt\vel=\nrg H.
\ee

In turn, (\ref{non 5}) becomes
\begin{eqnarray}
\frac{\partial ^2u}{\partial t^2}
&=&\nrg H w+\vel\pt w
+\left(\vel w_i-\frac{\partial ^2u}{\partial x^i\partial t}\right)\dt x^i
\nonumber\\
&=&\nrg H w+\vel\pt w
+\frac{\vel}{w}\left(\frac{\partial ^2u}{\partial x^i\partial t}-\vel w_i\right)\delta^{ij}u_j. 
\nonumber
\end{eqnarray}
Since $w_i=\frac{1}{w}\,\delta^{kl}u_ku_{li}$ and
$\pt w=\frac{1}{w}\,\delta^{ij}u_j\frac{\partial ^2u}{\partial x^i\partial t}$, 
we obtain
\be \nonumber
\frac{\partial ^2u}{\partial t^2}
=\nrg H w+2\frac{\vel}{w}\,\delta^{ij}\frac{\partial ^2u}{\partial x^i\partial t}
\,u_j-\frac{\vel^2}{w^2}\,\delta^{ki}\delta^{lj}u_{kl}u_iu_j.
\ee 
In addition, from (\ref{non 4}) we have
\be \label{non 7}
\frac{\vel}{w}=\frac{u_t}{1+|Du|^2}
\ee 
and from (\ref{non 1})
\be \label{non 8}
\nrg=\frac{1}{2}\left(\left|\frac{ dF}{dt}\right|^2+n\right)
=\frac{1}{2}\left(\vel^2+n\right)=\frac{1}{2}\left(\frac{u_t^2}{1+|Du|^2}+n\right).
\ee 
Moreover, the second fundamental form is given by
\be \nonumber
\dd hij=\left\langle F_{ij},\nu\right\rangle =\frac{\dd uij}{w}
\ee 
and the induced metric and its inverse by
$$
\dd gij=\dd\delta ij+u_iu_j, \quad\uu gij=\uu\delta ij-\frac{1}{w^2}\uu\delta ik\uu\delta jl u_ku_l, 
$$  
so that
\be \label{non 9}
Hw=\uu gij\dd uij=\left(\uu\delta ij-\frac{1}{w^2}\uu\delta ik\uu\delta jl u_ku_l\right)\dd uij.
\ee 
This leads us to the equation \eqref{eq graph} for graphs over a flat subspace $Z^\perp$.

\noindent{\sl Step 2: Hyperbolicity of the second-order equation.} 
The equation \eqref{eq graph} is hyperbolic if and only if the following matrix
$$\left(A^{\alpha\beta}\right)_{\alpha,\beta=0,\dots,n}:=\begin{pmatrix}
-1&\frac{\sigma}{w}\,u_1&\cdots&\frac{\sigma}{w}\,u_n\\
\frac{\sigma}{w}\,u_1& &\\
\vdots& & \left(A^{ij}\right)_{i,j=1,\dots, n}&\\
\frac{\sigma}{w}\,u_n& & \\
\end{pmatrix}$$
with
$$A^{ij}:=(e+\sigma^2)\uu gij-\sigma^2\uu\delta ij, \quad i,j=1,\dots,n$$
satisfies
$$d_k:=\det\left(\left(A^{\alpha\beta}\right)_{\alpha,\beta=0,\dots, k}\right)<0, \quad
\,\,0\le k\le n.
$$
Namely, fix any point $p\in Z^\perp$, and 
choose an orthonormal basis $e_1,\dots,e_n$ spanning $Z^\perp$ such that
at $p$ 
$$
Du(p)=u_1e_1, \quad u_i=0, \quad\,\, 2\le i\le n.
$$
In this basis the matrix $A$ at $p$ takes the (symmetric) form
$$\left(A^{\alpha\beta}\right)_{\alpha,\beta=0,\dots,n}=\begin{pmatrix}
-1                    & \frac{\sigma}{w}\,u_1         &0      & 0       & \cdots  & 0      \\[2mm]
\frac{\sigma}{w}\,u_1 & \frac{e-|Du|^2\sigma^2}{w^2}  &0      & 0       & \cdots  & 0      \\[2mm]
0                     & 0                             &e      & 0       & \cdots  & 0      \\[0mm]
0                     & 0                             &0      & \ddots  & \ddots  & \vdots \\[0mm]
\vdots                & \vdots                        &\vdots & \ddots  & \ddots  & 0      \\[0mm]
0                     & 0                             &0      & \cdots  & 0       & e      \\[2mm]
\end{pmatrix}$$
This gives in view of $u_1^2=|Du|^2$
$$
d_0=-1,\quad \quad d_k=-\,\frac{e^k}{w^2}<0\,\quad\quad k\ge 1.
$$
Since $p$ was arbitrary, this proves that equation (\ref{eq graph}) is indeed hyperbolic.

\noindent{\sl Step 3: First-order formulation.} By introducing the first-order variables 
$\sigma = u_t / w$ and $b_i:= u_i$, and regarding $w$ as a function of $b=(b_i)$,   
we can rewrite \eqref{eq graph} in the form
$$
\aligned 
& {\del \sigma \over \del t} 
       - \sigma {b_i \over w(b)} \, \delta^{ij} {\del \sigma\over \del x^j}  
       -  e(\sigma) \, {g^{ij}(b) \over w(b)}
       {\del b_i\over \del x^j}  
       = 0,
\\
& {\del b_i \over \del t} - w(b) \, {\del \sigma \over \del x^i} - \sigma {u_l \over w(b)} \delta^{kl} {\del b_k \over \del x^i} = 0. 
\endaligned 
$$
This is a first order nonlinear system in $\sigma,b$, which can be checked to take the desired conservative form \eqref{eq graph2}. 

Considering the function $E=(\sigma,b) $ introduced in the theorem, we also compute 
$$
\aligned 
& {\del^2 E \over \del \sigma^2} = w, 
\\ 
& {\del^2 E \over \del \sigma \del b_k} = w^{-1} \sigma b_i, 
\\
& {\del^2 E \over \del b_j \del b_k} = w^{-1} {1 \over 2} (\sigma^2+n) \left( \delta_{jk} - {b_j b_k \over w^2} \right), 
\endaligned 
$$
which is a non-negative matrix, since for all scalar $Y$ and vector $X=(X^j)$ (not both zero) 
the Hessian evaluated at $(Y,X)$ equals    
$$
w \Big( Y + \sigma {b_j \over w^2} X^j \Big)^2 - {\sigma^2 \over w^3} \Big(b_i X^j \Big)^2 
+  w^{-1} {1 \over 2} (\sigma^2+n) \left( X^j X^k \delta_{jk} - {X^j b_j X^k b_k \over w^2} \right), 
$$
or equivalently 
$$ 
\aligned 
& w \Big( Y + \sigma {b_j \over w^2} X^j \Big)^2 
+  w^{-3} {1 \over 2} \sigma^2 \left( (X^j X^k \delta_{jk})  - 2 \Big(b_i X^j \Big)^2 \right)
\\
& +  w^{-3} {n \over 2} \, X^j X^k \delta_{jk} 
 +  w^{-3} {1 \over 2} (\sigma^2 +n) \left( (X^j X^k \delta_{jk}) ( b_i b_l \delta^{il})  - (X^j b_j)^2 \right).
\endaligned 
$$
This expression is positive if and only if we impose the restriction $b_i b_j \delta^{ij}< 1/2$.  (Each term 
in the above decomposition has a positive sign.)

Since the equation \eqref{eq graph2} has the form of a system of conservation laws and admits a convex entropy, 
it can be put in a symmetric hyperbolic form. Indeed, introducing the variables 
$$
a := u_t = \sigma w(b),  \qquad   c_i:= e(\sigma) {b_i \over w(b)}, 
$$
which is nothing but the gradient of $E$, we obtain 
\be 
\label{eq graph3}
\aligned 
& {\del \sigma \over \del t} 
       - {\del \over \del x^j}  \Big( c_i\, \delta^{ij} \Big) = 0,
\\
& {\del b_i \over \del t} - {\del \over \del x^i} (\sigma w) = 0.
\endaligned 
\ee
Then, by expressing (implicitly) $\sigma$ and $b_i$ as functions $\overline \sigma, \overline b_i$ 
of the new unknowns $a$ and $c_i$, one can check that 
the above system is symmetric, in the sense that
$$
{\del \overline \sigma \over \del c^i}(a,c) = {\del \overline b_i \over \del a}(a,c).
$$
In turn, the system is locally well-posed in $H^s$ with $s >1+n/2$. 
\end{proof}


\subsection*{Weak solutions to the normal flow of graphs}

To define weak solutions we rely on the conservative form exhibited in \eqref{eq graph2}. 

\begin{definition} A Lipschitz continuous map $u: [0,T) \times \R^n \to \R$ is called
a {\sl weak solution} to the HMCF' equation for graphs \eqref{eq graph} if and only if for every 
test-function $\theta : [0,T) \times \R^n \to \R$ (that is, compactly supported $C^\infty$ functions) 
$$
\int_{(0,T)} \int_{\R^n} \left( 
{\del u \over \del t}  {\del \theta \over \del t} 
- {1 \over 2} \left( \left(w(Du)\right)^{-2}  \, \left|{\del u \over \del t}\right|^2 +n \right) 
 {\del u \over \del x^i}  {\del \theta \over \del x^j} 
\delta^{ij} \right) \left( w(Du) \right)^{-1} 
\, dx dt = 0. 
$$
where $w(Du)^2 = \left( 1 + \delta^{ij} {\del u \over \del x^i}  {\del u \over \del x^j}  \right)$. 
It is called an {\sl entropy solution} if, moreover, the inequality 
$$
\int_{(0,T)} \int_{\R^n} {1 \over 2} \left( \left(w(Du)\right)^{-2}  \, \left|{\del u \over \del t}\right|^2 +n \right)  \,  w(Du)  \, \theta \, dx dt \leq 0
$$
for every non-negative test-function $\theta$. 
\end{definition} 

We have the following uniqueness result, which relies on the fact that the energy is convex in the variables 
$\sigma,b$.

\begin{theorem}[Uniqueness of classical solutions within the class of entropy solutions]
 Given $\eps>0$, there exists a constant $C_\eps$ such that the following property holds.
Let $u$ be a Lipschitz continuous entropy solution and $u'$ be a solution of class $\Ccal^2$, 
both being defined up to some time $T>0$ and satisfying the uniform hyperbolicity
condition
$$
|Du|^2 < {1-\eps \over 2}, \qquad |Du'|^2 < {1-\eps \over 2}. 
$$
Then, provided the Lipschitz norm of $u$ if less than $C_\eps$
and the $C^2$ norm of $u'$ is less than $C_\eps$ then for all times $t \in [0,T)$ 
$$
\aligned 
& \int_{\R^n} \left(\left|  {\del u \over \del t} - {\del u' \over \del t} \right|^2 +  \left| Du-Du' \right|^2  \right)(t,x) \, dx
\\
& \leq C_\eps e^{C_\eps \, t} \, 
\int_{\R^n} \left(\left|  {\del u \over \del t} - {\del u' \over \del t} \right|^2 +  \left| Du-Du' \right|^2  \right)(t,x) \, dx.
\endaligned 
$$
\end{theorem} 

\begin{proof} Under the assumptions made in the theorem, consider the expression 
$$
\aligned
Q & = Q(t,x) = E(\sigma, b) - E(\sigma', b') - {D E \over D(\sigma,u)} (\sigma', b') ( (\sigma,u) - (\sigma',u')) 
\\
& = \Big( (1 + |Du|^2)^{-1} |u_t|^2 +n\Big) \sqrt{1 + |Du|^2} 
           - \Big( (1 + |Du'|^2)^{-1} |u_t'|^2 +n\Big) \sqrt{1 + |Du'|^2} 
\endaligned 
$$ 
and note that for some constant $C_\eps^{1}>0$ 
$$
{1 \over C_\eps^{1}} \, \big( |u_t - u_t'|^2 + |Du - Du'|^2 \big) 
\leq Q \leq 
C_\eps^{1} \, \big( |u_t - u_t'|^2 + |Du - Du'|^2 \big).   
$$
On the other hand, a direct calculation using the fact that $u$ is an entropy solution
and $u'$ is a classical solution yields the inequality 
$$
{d \over d t} \int_{\R^n} Q(t,x) \, dx \leq C_\eps^{2} \, \int_{\R^n} Q(t,x) \, dx,
$$
where the constant $C_\eps^{2}$ depends upon up to second-order derivatives of the solution $u'$. 
The conclusion follows from Gronwall's inequality and the fact that $Q$ is comparable with
$|u_t - u_t'|^2 + |Du - Du'|^2$. 
\end{proof}


\subsection*{Global existence of one-dimensional graphs}

In view of Theorem~\ref{P-graph}, the hyperbolic mean curvature flow equation in the case $n=1$ reads 
\begin{eqnarray}
\label{non 10}
u_{tt}
&=&\frac{1}{2}\left(\frac{u_t^2}{1+u_x^2}+1\right)\frac{u_{xx}}{1+u_x^2}
+2\frac{u_xu_t\dd uxt}{1+u_x^2}-\frac{u_x^2u_t^2u_{xx}}{(1+u_x^2)^2}\nonumber\\
&=&\frac{u_{xx}}{2(1+u_x^2)^2}\left(u_t^2+1+u_x^2-2u_x^2u_t^2\right)
+2\frac{u_xu_t\dd uxt}{1+u_x^2}.\nonumber
\end{eqnarray}
Therefore, we find 
\be 
\label{non 11}
u_{tt}=\frac{1+u_x^2+u_t^2-2u_x^2u_t^2}{2(1+u_x^2)^2}\,u_{xx}
+2\frac{u_xu_t}{1+u_x^2}\,\dd uxt.
\ee 
This equation written on the real line with initial data 
$$
u(x,0) = u_0(x), \qquad u_t(0,x) = u_1(x), \qquad x \in \R, 
$$
describes the vibrations of an infinitely long string, 
with initial position $u_0$ and initial velocity $u_1$.  

Relying on the definition of weak solutions introduced earlier for general dimensions, 
we now prove: 

\begin{theorem} [Global existence of weak solutions] 
There exists a constant $\delta_0>0$ such that given any initial data $u_0, u_1: \R \to \R$ such that 
$$
TV(u_{0,x}) + TV(u_1) < \delta_0,   
$$
the initial-value problem for the equation (HMCF') in the second-order form \eqref{non 11}
admits an entropy solution $u=u(t,x)$ such that the functions $u_x$ and $u_x$ have bounded 
variation in space, uniformly in time.   
\end{theorem}

\begin{proof} 
Since
\be \nonumber
\text{det}\begin{pmatrix}-1& &\frac{u_xu_t}{1+u_x^2}\\
\frac{u_xu_t}{1+u_x^2}& & \frac{1+u_x^2+u_t^2-2u_x^2u_t^2}{2(1+u_x^2)^2}\end{pmatrix}
=-\frac{1+u_x^2+u_t^2}{2(1+u_x^2)^2}=-\,\frac{e}{w^2}<0, 
\ee 
this equation is always hyperbolic. 
We introduce the variables
$$
a:=\frac{u_t}{w}, \qquad b:=u_x.
$$
These are both conservative quantities. For $b$ this is trivial since
\be 
\label{non 12}
b_t = u_{xt} = (u_t)_x = (a\sqrt{1+b^2})_x.
\ee 
For $a$ we may either see this directly from the conservation law (\ref{cont}) with 
$Y=\frac{\partial}{\partial u}$ or we may compute
$$
\aligned 
a_t 
& = \frac{u_{tt}}{w}-\frac{u_xu_tu_{xt}}{w^3}
\\
& = \frac{1+u_x^2+u_t^2-2u_x^2u_t^2}{2w^5}\,u_{xx}+2\frac{u_xu_t}{w^3}\,u_{xt}-
\frac{u_xu_tu_{xt}}{w^3}
\\
& =  \frac{1+u_x^2+u_t^2-2u_x^2u_t^2}{2w^5}\,u_{xx}+\frac{u_x\left(u_t^2\right)_x}{2w^3}
\\
& = \frac{1+u_x^2+u_t^2-2u_x^2u_t^2}{2w^5}\,u_{xx}
+\left(\frac{u_xu_t^2}{2w^3}\right)_x-u_t^2\left(\frac{u_x}{2w^3}\right)_x, 
\endaligned 
$$
thus 
$$
\aligned 
a_t & =  \frac{u_{xx}}{2w^3}+\left(\frac{u_xu_t^2}{2w^3}\right)_x
\\
& = \left(\frac{u_x}{2w}+\frac{u_xu_t^2}{2w^3}\right)_x
 =\left(\frac{u_x(1+u_x^2+u_t^2)}{2w^3}\right)_x
 = \left(\frac{\left(1+a^2\right)b}{2\sqrt{1+b^2}}\right)_x.
\endaligned 
$$

This can be rewritten in the form
$$
a_t=\frac{ab}{\sqrt{1+b^2}}\,a_x+\frac{1+a^2}{2(1+b^2)^{3/2}}\,b_x
$$
and equation (\ref{non 12}) is equivalent to
$$
b_t=\sqrt{1+b^2}\,a_x+\frac{ab}{\sqrt{1+b^2}}\,b_x.
$$
Combining the last two equations gives the system
\be 
\label{non 14}
\begin{pmatrix}a\\ b\end{pmatrix}_t
- A(a,b) \, \begin{pmatrix}a\\b\end{pmatrix}_x
=0,
\ee 
where  
\be 
\label{non 15}
A(a,b) =
\frac{1}{\sqrt{1+b^2}}\begin{pmatrix}ab& & \displaystyle\frac{1+a^2}{2(1+b^2)}\\[2mm]
1+b^2& & ab
\end{pmatrix}. 
\ee 

Recall that a necessary and sufficient condition for a quantity $\eta=\eta(a,b)$ 
to be a conserved quantity is for 
\be 
\nonumber
\begin{pmatrix}
\eta_{aa}& & \eta_{ab}\\[2mm]
\eta_{ab}& & \eta_{bb}
\end{pmatrix}\, A
\ee 
to be a symmetric matrix. For the nonlinear hyperbolic system under consideration this gives
\be \nonumber
ab \, \eta_{ab}+(1+b^2) \, \eta_{bb}=\frac{1+a^2}{2(1+b^2)} \, \eta_{aa}+ab \, \eta_{ab},
\ee 
that is, 
\be 
\label{non 16}
\eta_{aa} = \frac{2(1+b^2)^2}{1+a^2}\,\eta_{bb}.
\ee 
This is clearly a linear hyperbolic equation. 
From this we compute for example that $ab=\frac{u_xu_t}{w}$ is a conserved quantity,
a fact which also follows directly from (\ref{cont}) with $Y=\frac{\partial}{\partial x}$.

As one easily computes, the two eigenvalues of $A$ are given by
\be \label{non 17}
\lambda_{\pm}=\frac{1}{\sqrt{1+b^2}}\left(ab\pm\sqrt{\frac{1+a^2}{2}}\right),
\ee 
while the eigenspaces are spanned by the vectors 
\be 
\label{non 18}
\mu_\pm=\begin{pmatrix}\pm\frac{\sqrt{1+a^2}}{\sqrt{2}(1+b^2)}\\ 1
\end{pmatrix}.
\ee 
Regarded as a function of $a$ and $b$ the gradient of $\lambda_\pm$ equals 
\be 
\label{non 19}
D\lambda_\pm=\frac{1}{\sqrt{1+b^2}}\begin{pmatrix}b\pm\frac{a}{2\sqrt{\frac{1+a^2}{2}}}
\\
a-\frac{b}{1+b^2}\left(ab\pm\sqrt{\frac{1+a^2}{2}}\right)\end{pmatrix},
\ee 
hence 
$$
\aligned 
& \langle D\lambda_\pm,\mu_\pm\rangle
\\
& = \frac{1}{\sqrt{1+b^2}} \Bigg(\pm\frac{\sqrt{1+a^2}}{\sqrt{2}(1+b^2)}
       \Big( b\pm\frac{a}{2\sqrt{\frac{1+a^2}{2}}} \Big)
       + \frac{a}{1+b^2}\mp\frac{b}{1+b^2}\sqrt{\frac{1+a^2}{2}}\Bigg) 
\\
& = \frac{3a}{2(1+b^2)^{3/2}}=\frac{3u_t}{2w^4}.
\endaligned 
$$ 
Hence, the hyperbolic 
system under consideration is {\sl not} genuinely nonlinear in the sense of Lax. 

However, we observe that the genuine nonlinearity is lost on a 
hypersurface (that is, $\big\{ u_t =0 \big\}$) which is itself non-degenerate, in the sense that 
$$
\langle D \langle D\lambda_\pm,\mu_\pm\rangle, \mu_\pm \rangle 
\neq 0 \quad \text{ along the hypersurface } u_t =0. 
$$
Therefore, we are in a position to apply to the system of conservation laws 
$$
\aligned 
& a_t - \left(\frac{\left(1+a^2\right)b}{2\sqrt{1+b^2}}\right)_x = 0, 
\\ 
& b_t - (a\sqrt{1+b^2})_x=0, 
\endaligned
$$  
the global existence theorem in Iguchi and LeFloch \cite{iguchilefloch}, 
which provides the existence of a solution with bounded variation when the initial data have small bounded variation.
\end{proof}


\section{Local-in-time existence}
\label{S-local} 

We now turn to the discussion of the existence of solutions to the normal
hyperbolic mean curvature flow (HMCF'), where $M_t$ cannot necessarily be written as an entire graph
over a flat subspace. We use standard notation and, in particular, denote by 
$H_\loc^s(M)$ the Sobolev space of 
locally squared integrable (tensor-valued) maps defined on $M$ 
whose all $s$-order derivatives (in one local chart and in the distributional sense) 
are also locally squared integrable.  

\begin{theorem} 
\label{exist}

Let $M$ be a smooth, orientable compact manifold with dimension $n$, 
and $\overline F : M \to \R^{n+1}$ be an immersion of $M$ in the Euclidian space.  
Given a (scalar) normal velocity field $\overline \sigma : M \to \R$ in the Sobolev space $H^{s+1}(M)$ 
with $s >1+n/2$, there exists a unique flow $F:[0,T)\times M \to \R^{n+1}$ 
in the space $L^\infty([0,T), H^{s+1}(M)) \cap Lip([0,T), H^s(M))$ 
which is defined on some maximal time interval and 
satisfies the normal hyperbolic mean curvature flow equation (HMCF'), 
together with the initial conditions
$$
F(0) = \overline F, \qquad {d F \over dt}(0) = \overline \sigma \, \nu(0). 
$$
\end{theorem}

We will provide two different arguments to handle the equation (HMCF').

Let us us cover the manifold with finitely many local charts, chosen in such a way that the
manifold can be viewed locally as a graph over its tangent plane at some point. In each local chart, we apply the local
existence theorem for graphs  established in Theorem~\ref{P-graph}. 
Indeed, due to the property of finite speed of propagation satisfied by hyperbolic equations, 
all of the arguments therein can be localized in space and apply in each coordinate patch.  
Then, by patching together these local solutions
and using the fact that only finitely many charts suffice to cover the manifold $M$, we can find a sufficiently 
small $T$ such that every local solution is defined within the time interval $[0,T)$, at least. 
This completes the proof of the theorem. 

The rest of this section is devoted to provide a second proof of Theorem~\ref{exist}
which is also of interest in its own sake. 
We will now express (HMCF') as a single scalar equation in terms of
a height function $u$ with respect to a fixed initial hypersurface.
To this end let us discuss the case of flows $F:[0,T) \times M\to\real{n+1}$
such that each $M_t:=F(t,M)$ is an entire graph over a
fixed reference manifold $\Sigma$ given by an immersion $G:M\to\Sigma\subset\real{n+1}$.
If each $M_t$ can be written as a graph over $\Sigma$, there must exist a 
family of smooth height functions $u: [0,T)\times M\to\real{}$ 
and a family $\xi:[0,T)\times M\to M$ of diffeomorphisms such that
\be \label{graph}
F(t,x)=G\bigl(\xi(t,x)\bigr)+u\bigl(t,\xi(t,x)\bigr)\normal
\bigl(\xi(t,x)\bigr),
\ee 
where
$\normal$ is the inward unit normal along $\Sigma$.
Let us denote the metric on $M=\Sigma$ by $\dd\sigma ij$ and the second
fundamental form by $\dd\tau ij$ The induced connection on $M$ with respect to
$\sigma$ will be denoted by $D$.

The tangent vectors take the form
\begin{eqnarray}
F_i(t,x)
&=&\Bigl(
G_j(\xi(t,x))+u_j(t,\xi(t,x))\normal(\xi(t,x))\nonumber\\
&&-u(t,\xi(t,x))\tau^k_j(\xi(t,x))G_k(\xi(t,x))\Bigr) \xi^j_i(t,x)\,\nonumber
\end{eqnarray}
where in this section a raised index will be raised with respect to the metric
$\sigma$, i.e. $\tau^k_i=\uu\sigma kl\dd\tau il$. It will be convenient
to define the following tensor
$$\dd Nij:=\dd\sigma ij-u\dd\tau ij.$$
Then the tangent vectors can be written in the form
\be \label{graph tangent}
F_i=\left(u_j\normal+N_j^lG_l\right)\xi^j_i
\ee 
and for the second derivative $F_{ij}:=\frac{\partial^2 F}{\partial x^i\partial x^j}$ we get
\begin{eqnarray}
F_{ij}
=\left( 
\left(u_{kl}+N_k^m\tau_{ml}\right)\normal+\left(D_lN_k^m-u_k\tau^m_l\right)G_m\right)
\xi^k_i\xi^l_j
+\left( u_k\normal+N_k^mG_m\right) \xi^k_{ij}.\nonumber
\end{eqnarray}
The induced metric tensor $\dd gij=\langle F_i,F_j\rangle$ is
\be \label{graph metric}
\dd gij=\left(u_ku_l+\dd Nkm N^m_l\right)\xi^k_i\xi^l_j.
\ee 
In the following we will assume that $u$ is sufficiently small, so
that the symmetric tensor $\dd Nij$ is invertible and we denote its inverse by
$\uu{\widetilde N}ij$. Let us define
$$w:=\sqrt{1+\widetilde N^k_i\uu{\widetilde N}ilu_ku_l}.$$
The inward unit normal along $M_t$ is then determined by
\be \label{graph normal}
\nu=\frac{1}{w}\left(\normal-\uu{\widetilde N}klu_kG_l\right), 
\ee 
so that
\begin{eqnarray}
\dd hij
&=&\left\langle F_{ij},\nu\right\rangle\nonumber\\
&=&\frac{1}{w}
\left(
u_{kl}+N_k^m\tau_{ml}+\tilde N^{rm}u_r\left(u_k\tau_{ml}-D_lN_{km}\right)\right) 
\xi^k_i\xi^l_j.
\end{eqnarray}

We need expressions for $\dt F$ and $\ddt F$.
From (\ref{graph}) we obtain
\begin{eqnarray}
\dt F
&=&G_k\dtx k+\left(u_t+u_k\dtx k\right)\normal-u\tau_k^l\dtx kG_l,\nonumber
\end{eqnarray}
where the subscript $t$ in $u_t$ denotes a partial derivative with respect to $t$, i.e.
$u_t=\frac{\partial u}{\partial t}$. Rearranging terms gives
\be \label{graph vel}
\dt F=\left(u_t+u_k\dtx k\right)\normal+N_k^l\dtx kG_l.
\ee 
This implies the relations
\be \label{graph v}
\vel=\left\langle\dt F,\nu\right\rangle=\frac{u_t}{w}
\ee 
and
\be \label{graph V}
\tvel_i=\left\langle\dt F,F_i\right\rangle=\left(u_ju_t+\dd {\tilde g}jk\dtx k\right)\xi^j_i, 
\ee 
where
$$
\tilde g_{kl}:=u_ku_l+N_l^mN_{ml}.
$$
We differentiate (\ref{graph vel}) with respect to time and compute
\begin{eqnarray}
\ddt F
&=&\left(\dd utt+2\dd utk\dtx k+u_k\frac{{d}^2\xi^k}{{d}t^2}\right)
\normal-\left(u_t+u_k\dtx k\right)\tau_i^l\dtx iG_l\nonumber\\
&&+D_iN_k^l\dtx i\dtx kG_l+N^l_k\frac{{d}^2\xi^k}{{d}t^2}G_l
+N^l_k\dd\tau il\dtx k\dtx i\,\normal.\nonumber
\end{eqnarray}
Hence, we have 
\begin{eqnarray}
\ddt F
&=&\left(\dd utt+2\dd utk\dtx k+u_k\frac{{d}^2\xi^k}{{d}t^2}
+N^l_k\dd\tau il\dtx k\dtx i\right)\normal\label{graph acc}\\
&&+\left(N^l_k\frac{{d}^2\xi^k}{{d}t^2}+D_iN_k^l\dtx i\dtx k
-\left(u_t+u_k\dtx k\right)\tau_i^l\dtx i\right)G_l 
\nonumber
\end{eqnarray}
and, consequently, 
\begin{eqnarray}
\acc&=&\left\langle\ddt F,\nu\right\rangle\nonumber\\
&=&\frac{1}{w}
\left(\dd utt+2\dd utk\dtx k+u_k\frac{{d}^2\xi^k}{{d}t^2}
+N^l_k\dd\tau il\dtx k\dtx i\right)\nonumber\\
&&-\frac{1}{w}\left(N^l_k\frac{{d}^2\xi^k}{{d}t^2}
+D_iN_k^l\dtx i\dtx k-\left(u_t+u_k\dtx k\right)\tau_i^l\dtx i\right)
\widetilde N^m_lu_m\nonumber
\end{eqnarray}
Therefore, we have 
\begin{eqnarray}
\acc&=&\frac{1}{w}\left(
\dd utt+2\dd utk\dtx k
+N^l_k\dd\tau il\dtx k\dtx i\right.\nonumber\\
&&\left.-D_iN_k^l\widetilde N^m_lu_m\dtx i\dtx k+\left(u_t+u_k\dtx k\right)
\tau_i^l\widetilde N^m_lu_m\dtx i\right)
\label{graph a}
\end{eqnarray}
and, moreover, 
\begin{eqnarray}
\tacc_i&=&\left\langle\ddt F,F_i\right\rangle\nonumber\\
&=&u_j\left(\dd utt+2\dd utk\dtx k+u_k\frac{{d}^2\xi^k}{{d}t^2}
+N^l_k\dd\tau sl\dtx k\dtx s\right)\xi^j_i\nonumber\\
&&+\dd Njl\left(N^l_k\frac{{d}^2\xi^k}{{d}t^2}+D_sN_k^l\dtx s\dtx k
-\left(u_t+u_k\dtx k\right)\tau_s^l\dtx s\right)\xi^j_i.\nonumber
\end{eqnarray}
Reordering gives the final formula 
\begin{eqnarray}
\tacc_i
&=&\left(
u_j\dd utt+2u_j\dd utk\dtx k+\dd {\tilde g}jk
\frac{{d}^2\xi^k}{{d}t^2}
+(u_jN^l_k\dd\tau sl+\dd NjlD_sN_k^l)\dtx s\dtx k\right.\nonumber\\
&&\left.-\dd Njl\tau^l_s\left(u_t+u_k\dtx k\right)\dtx s\right) \xi^j_i.
\label{graph A}
\end{eqnarray}

We summarize our results in the following proposition.

\begin{proposition}\label{prop summ}
If $F:[0,T)\times M\to\real{n+1}$ is an arbitrary flow, where each $M_t=F(t,M)$ is
represented as a graph over $G:M\to\Sigma$ as above, then $(u,(\xi^k)_{k=1,\dots,n})$ is
a solution of the coupled system
\begin{eqnarray}
\vel&=&\frac{u_t}{w}\label{g1}\\
\tvel_i&=&\left(u_ju_t+\dd {\tilde g}jk\dtx k\right)\xi^j_i\label{g2}\\
\acc&=&\frac{1}{w}\left(
\dd utt+2\dd utk\dtx k
+N^l_k\dd\tau il\dtx k\dtx i\right.\nonumber\\
&&\left.-D_iN_k^l\widetilde N^m_lu_m\dtx i\dtx k+\left(u_t+u_k\dtx k\right)
\tau_i^l\widetilde N^m_lu_m\dtx i\right)
\label{g3}\\
\tacc_i
&=&\left(
u_j\dd utt+2u_j\dd utk\dtx k+\dd {\tilde g}jk
\frac{{d}^2\xi^k}{{d}t^2}
+(u_jN^l_k\dd\tau sl+\dd NjlD_sN_k^l)\dtx s\dtx k\right.\nonumber\\
&&\left.-\dd Njl\tau^l_s\left(u_t+u_k\dtx k\right)\dtx s\right) \xi^j_i.
\label{g4}
\end{eqnarray}
\end{proposition}

Note, that up to this point we have not chosen a particular flow. This will be done in the
next step. For (HMCF'), we have $\tvel_i=0$, $\alpha=\nrg H$, $\tacc=-\nabla_i\nrg $.
From (\ref{g2}) we obtain
\begin{equation}\label{g2b}
\dtx k=-\uu{\tilde g}jku_ju_t, 
\end{equation}
where $\uu{\tilde g}jk$ denotes the inverse of $\dd{\tilde g}jk$.
Inserting this into (\ref{g3}) we get
\begin{eqnarray}
\nrg H&=&\frac{1}{w}\left\{\dd utt-2\uu{\tilde g}jku_ju_t\dd utk
+L\right\}, \label{hyp}
\end{eqnarray}
where $L$ is a term of first order in $u$. This, by the compactness of $\Sigma$,
clearly is a uniformly hyperbolic equation, provided
$u$ is sufficiently small in $C^1(\Sigma)$. So we can provide a short-time solution for (HMCF'),
if we assume $u_{|t=0}=0$, i.e. if $\Sigma=M_0.$ After having solved (\ref{hyp}) we can
solve (\ref{g2b}) for $\dtx k$. If $u$ is sufficiently small in $C^1(\Sigma)$, then
$\dtx k$ will be small as well so that by choosing $\xi_{t=0}=\operatorname{Id}_{\Sigma}$
we obtain a family of diffeomorphisms $\xi(t)$ solving (\ref{g4}).
This completes the proof of Theorem~\ref{exist}.


\section{Finite-time blow-up results}
\label{S-blow-up} 

\subsection*{An example of finite-time blow-up} 

The hyperbolic mean curvature flow  may blow-up in finite time in a way that 
is completely analogous to the standard mean curvature flow. We provide here a typical example.

Let $F_0:S^n\to\real{n+1}$ be a round sphere of radius $r_0$. If
$F:[0,T)\times S^n\to\real{n+1}$ is a solution of (HMCF') with
initial data $F(0,x)=F_0(x)$ and $\dt F(0,x)=\vel_0\nu(x)$ with a constant 
$\vel_0\in\real{}$, then $M_t:=F(t,S^n)$ is a concentric sphere with radius 
$r(t)$. In this case, the \flow\ reduces to the ordinary differential equation (ODE) 
\begin{eqnarray}
r\ddot r+\frac{n}{2}(\dot r)^2+\frac{n^2}{2}&=&0,\nonumber\\
r(0)&=&r_0,\nonumber\\
\dot r(0)&=&-\vel_0.\nonumber
\end{eqnarray}
This second-order equation can be reduced to the following first-order ODE
\begin{eqnarray}
\dot r&=&\begin{cases}
\disp\phantom{-\,}
\sqrt{(n+\vel_0^2)\left(\frac{r_0}{r}\right)^n-n},&\text{ if }
\vel_0<0,\\ \\
\disp-\,\sqrt{(n+\vel_0^2)\left(\frac{r_0}{r}\right)^n-n},&\text{ if }
\vel_0\ge0,
\end{cases}\label{ode}\\
r(0)&=&r_0.\nonumber
\end{eqnarray} 

The solution depends upon the dimension $n$. In the case $n=1$, 
we obtain the cycloid
\begin{eqnarray}
&&r\sqrt{\frac{c}{r}-1}+c
\arctan{\sqrt{\frac{c}{r}-1}}
\nonumber\\
&=&\begin{cases}
\disp-t-r_0\vel_0-c\arctan{\vel_0},
&\text{ if }
\vel_0<0,\\ \\
\disp t+r_0\vel_0+c\arctan{\vel_0},
&\text{ if }
\vel_0\ge0,
\end{cases}\nonumber
\end{eqnarray}
where $c=r_0(1+\vel_0^2)$.

On the other hand, in the case $n=2$, we obtain the explicit solution
\be \nonumber
r(t)=\sqrt{r_0^2-2r_0\vel_0t-2t^2}.
\ee 
If $\vel_0<0$, the sphere begins to expand until it starts to shrink
and eventually collapses to a point 
in a finite time $T$, given by   
$$
T=\frac{r_0}{2}\left(-\vel_0+\sqrt{\vel_0^2+2}\right).
$$

In the above situation, one can avoid the formation of singularities by rescaling the metric
according to its volume. 


\subsection*{Blow-up estimates based on the mean and total mean curvature} 

In some situations it is possible to derive blow-up results from the behavior of
the mean or total mean curvature of the system. To this end let us define the function
$$
\gamma(\sigma):=\frac{2}{\sqrt{n}}\arctan\left(\frac{\vel}{\sqrt{n}}\right),
$$
which satisfies 
$$
\gamma'=\frac{2}{n}\frac{1}{\frac{\vel^2}{n}+1}=\frac{1}{\nrg}
$$
and
$$
\dt\gamma=\gamma'\dt\vel=H.
$$
Given $U\subset M$ we define
$$
f_U(t):=\int\limits_U (\vel+\gamma\nrg)\,d\mu_t, 
\qquad 
E_U:=\int\limits_{U}e d\mu_t.
$$
Note that $\dt(\nrg d\mu_t)=0$ implies, that $E_U$ does not depend on $t$.

\begin{proposition}\label{prop bu1}
For any integrable $U\subset M$ and any $0\le t_1\le t_2\le T$ one has 
\begin{equation}\label{bu1}
\left|\int\limits_{t_1}^{t_2}\int\limits_UHd\mu_tdt\right|=\frac{1}{n}|f_U(t_2)-f_U(t_1)|
\le\frac{2\pi}{n\sqrt{n}}\,E_U
\end{equation}
and
\begin{eqnarray}
\left|\int\limits_{t_1}^{t_2}\int\limits_UH\nrg\, d\mu_tdt\right|
&=&\left|f_U(t_2)-f_U(t_1)+\int\limits_U\vel d\mu_{t_1}-\int\limits_U\vel d\mu_{t_2}\right|\nonumber\\
&\le&\frac{2(\pi+1)}{\sqrt{n}}\,E_U\label{bu2}.
\end{eqnarray}
In particular, if $T=\infty$, then for any $\epsilon>0$ and any choice of integrable
$U\subset M$ there exists a 
sequence $t_k\to \infty$ such that
$$\left|\int\limits_UHd\mu_{t_k}\right|<\epsilon, 
\qquad k\in\natural{}$$
and
$$\left|\int\limits_UH\nrg\,d\mu_{t_k}\right|<\epsilon, \qquad k\in\natural{}.
$$
\end{proposition}

\begin{proof}
Since $\dt\vel=\nrg H$ and $\dt\, d\mu_t=-\vel Hd\mu_t$ we compute
$$\dt f_U(t)=\int\limits_U(\nrg H+H\nrg+\gamma\vel\nrg H-(\vel+\gamma\nrg)\vel H) d\mu_t=n\int\limits_UHd\mu_t.
$$
The function
$\frac{\vel}{\nrg}+\gamma$ is a monotone increasing function in $\vel$ and 
$$-\frac{\pi}{\sqrt{n}}<\frac{\vel}{\nrg}+\gamma<\frac{\pi}{\sqrt{n}}.
$$
Therefore the function 
$$f_U(t)=\int\limits_U\left(\frac{\vel}{\nrg}+\gamma\right)\nrg d\mu_t$$
satisfies
\begin{equation}\label{bu3}
-\frac{\pi}{\sqrt{n}}\,E_U<f_U(t)<\frac{\pi}{\sqrt{n}}\,E_U.
\end{equation}
Then we obtain 
$$\dt f_U(t)=n\int\limits_U Hd\mu_t$$
which implies (\ref{bu1}).
From the observation 
$$0\le\int_U(\vel\pm\sqrt{n})^2d\mu_t=\int\limits_U(\vel^2+n\pm 2\sqrt{n}\vel)d\mu_t=2E_U
\pm 2\sqrt{n}\int\limits_U\vel d\mu_t$$
we conclude that 
\begin{equation}\label{bu4}
\left|\int\limits_U\vel d\mu_t\right|\le \frac{E_U}{\sqrt{n}}.
\end{equation}
Moreover, we have 
$$\dt\int\limits_U\vel\,d\mu_t=\int\limits_U(\nrg-\vel^2)Hd\mu_t
=-\int\limits_U\nrg Hd\mu_t+n\int\limits_UHd\mu_t,
$$
and thus 
$$
\dt\left(f_U(t)-\int\limits_U\vel\,d\mu_t\right)=\int\limits_U\nrg Hd\mu_t.$$
This and (\ref{bu3}), (\ref{bu4}) imply (\ref{bu2}).
\end{proof}

We can do even better:

\begin{proposition}
For any $x\in M$ and any $0\le t_1\le t_2\le T$ one has 
$$ 
\left|\int\limits_{t_1}^{t_2}H(x,t)dt\right|
=|\gamma(t_2)-\gamma(t_1)|\le\frac{2\pi}{\sqrt{n}}.
$$ 
In particular, if $T=\infty$, then for any $\epsilon>0$ and any $x\in M$ there exists a 
sequence $t_k\to \infty$ such that
$$|H(x,t_k)|<\epsilon, \qquad  k\in\natural{}.$$
\end{proposition}
\begin{proof}
This follows directly by integrating
$$\dt\gamma=H$$
and from
\begin{equation}\label{bu6}
|\gamma|\le\frac{\pi}{\sqrt{n}}.
\end{equation}
\end{proof}

\begin{proposition}
Consider the flow associated with a closed curve $C\subset\real{2}$ with non-vanishing rotation number $\chi(C)$.
Then $T<\infty$.
\end{proposition}

\begin{proof}
The rotation number of a curve is given by
$$\chi(C)=\frac{1}{2\pi}\int\limits_C H \, d\mu.
$$
This is a topological invariant, hence in particular $\dt\chi(C_t)=0$ for all smooth
deformations $C_t$ of $C$. It follows that
$$
\int\limits_C H \, d\mu_t
$$ 
is a (non-zero) constant. Hence, by Proposition \ref{prop bu1} we must have $T<\infty$.
\end{proof}


\bibliographystyle{amsplain}

\end{document}